\documentclass{amsart}
\numberwithin{equation}{section}

\newtheorem{thm}{Theorem}[section]
\newtheorem{lem}[thm]{Lemma}
\newtheorem{cor}[thm]{Corollary}

\newtheorem{defin}[thm]{Definition}
\newtheorem{rem}[thm]{Remark}

\begin{document}
\title{Proceedings of the American Mathematical Society}

\author{Ravshan Ashurov}
\address{Institute of Mathematics, Uzbekistan Academy of Science}
\curraddr{Institute of Mathematics, Uzbekistan Academy of Science,
Tashkent, 81 Mirzo Ulugbek str. 100170} \email{ashurovr@gmail.com}

\small

\title[Generalized localization]
{Generalized localization for spherical partial sums of the
multiple Fourier series and integrals}

\begin{abstract}

It is well known, that Luzin's conjecture has a positive solution
for one dimensional trigonometric Fourier series and it is still
open for the spherical partial sums $S_\lambda f(x)$, $f\in
L_2(\mathbb{T}^N)$, of multiple Fourier series, while it has the
solution for square and rectangular partial sums. Historically
progress with solving Luzin's conjecture has been made by
considering easier problems. One of such easier problems for
$S_\lambda f(x)$ was suggested by V. A. Il'in in 1968 and this
problem is called the generalized localization principle. In this
paper we first give a short survey on convergence
almost-everywhere of Fourier series and on generalized
localization of Fourier integrals, then present a positive
solution for the generalized localization problem  for $S_\lambda
f(x)$.

\end{abstract}

\maketitle

\section{Introduction }

Let $\{f_n\}$, $n\in \mathbb{Z}^N$, be the Fourier coefficients of a
function $f\in L_2(\mathbb{T}^N)$, $N\geq2$, i.e.
$$f_n=(2\pi)^{-N}\int\limits_{\mathbb{T}^N}
f(y)e^{-iny}dy,$$
where $\mathbb{T}^N$ is $N$-dimensional torus: $\mathbb{T}^N = (\pi, \pi]^N$. Consider the spherical partial
sums of the multiple Fourier series:
\begin{equation}\label{SL}
S_\lambda f(x)=\sum\limits_{|n|^2< \lambda}f_n\,e^{inx},
\end{equation}
where
$nx=n_1x_1+n_2x_2+...+n_Nx_N$
and $|n|=\sqrt{n_1^2+n_2^2+...+n_N^2}$.

The aim of this paper is to investigate convergence
almost-everywhere (a.e.) of these partial sums. One of the first
questions which arises in the study of a.e. convergence of the
sums (\ref{SL}) is the question of the validity of the Luzin
conjecture: is it true that the spherical sums (\ref{SL}) of
$N$-fold Fourier series of an arbitrary function $f\in
L_2(\mathbb{T}^N)$ converge a.e. on $\mathbb{T}^N$? In other
words, does Carleson's theorem extend to spherical partial sums
(\ref{SL})? The answer to this question is unknown so far (see
\cite{AAP} and \cite{AIN}). What is known is only that Hunt's
theorem does not extend to partial sums (\ref{SL}) \cite{MN}.
Namely, for each $p\in [1, 2)$ there exists a function $f\in
L_p(\mathbb{T}^N)$ such that on a set of positive measure
$$
\overline{\lim\limits_{\lambda\to\infty}}|S_\lambda
f(x)|=+\infty.
$$
As the authors of paper \cite{MN} showed, this result is a
consequence of Fefferman's celebrated theorem to the effect that
partial sums (\ref{SL}) does not converge in $L_p(\mathbb{T}^N)$,
$1\leq p <2$.

Historically progress with
solving the Luzin conjecture has been made by considering easier
problems. One of such easier problems is the generalized principle of localization.

Il'in \cite{IL} was the first to introduce the concept of
generalized localization principle for an arbitrary eigenfunction
expansions.  We present the corresponding definition for Fourier
series; for Fourier integrals the definition is similar.

\begin{defin}We say that for $S_\lambda f$ the generalized
localization principle in function class $L_p(\mathbb{T}^N)$, $p\geq 1$, holds true, if for
any function $f \in L_p(\mathbb{T}^N)$, with $f=0$ on a open set $\Omega\subset \mathbb{T}^N$ the equality
$$
\lim_{\lambda \to \infty} S_\lambda f(x) =  0
$$
is valid for a.e. on $\Omega$.
\end{defin}

Recall that the validity of the classical Riemann  localization
principle means that the convergence at a point $x_0$ of the Fourier series of a
function $f$ depends only on the behavior of $f$ in a small neighborhood
of that point. More exactly, if $f=0$ in an open set $\Omega\subset \mathbb{T}^N$ then the Fourier
series of $f$ converges to $0$ at every point of $\Omega$. For generalized
localization one requires that the convergence to $0$ occurs a.e. on $\Omega$.

Why one should consider the generalized localization instead of the classical
one in $L_p$-classes? According to the classical Riemann theorem, in the one-dimensional case, the localization
principle is valid for any integrable function. The situation changes in the case of
functions of two or more variables. In this case, there are examples of functions
with high smoothness for which the classical principle of localization is not valid
(see \cite{AAP}). Therefore, the principle of generalized localization, introduced by
Il'in, may become a replacement for the classical principle of localization.

Recently the problem of generalized localization was completely
solved in $L_p(\mathbb{T}^N)$ classes \cite{AR}. In the present
paper we first give a short survey on convergence a.e. of Fourier
series, then indicate  a sketch of the proof from \cite{AR}.

\section{Convergence almost-everywhere of Fourier series }

{\bf 1. One dimensional Fourier series.} A problem of expanding of a function $f(t)\in L_1(\mathbb{T})$, $\mathbb{T}=(-\pi, \pi]$, in terms of trigonometric
series
$$
\sum\limits a_k\cos kt+b_k\sin kt,
$$
arises in mathematical physics and investigation of the
convergence problems of such series is very important. Here $a_k$ and $b_k$ are as usual
Fourier coefficients of function $f$.

Partial sums are defined as follows:
$$
S_m(t)=\sum\limits_{k=0}^m a_k\cos kt+b_k\sin kt.
$$
If we define a new coefficients as $c_k=\frac{1}{2}(a_k-ib_k),c_{-k}=\overline{c_k},$ then partial sums can be rewtitten as
\begin{equation}\label{ps1}
S_m(t)=\sum\limits_{|k|\leq m}c_ke^{ikt}.
\end{equation}
These partial sums can easily be generalized in multidimensional case.

In 1915 N.N.Luzin made the conjecture that the Fourier series of
any function in $L_2(\mathbb{T})$ converges a.e.. In other words,
if
$$
\sum\limits|c_k|^2<\infty
$$
then the series
$$
\sum\limits c_ke^{ikt}
$$
converges for almost every $t\in \mathbb{T}=(-\pi,\pi].$

More than 50 years the Luzin's conjecture attracted the attention of the
specialists inducing a large number of investigations. In
1922 A.N.Kolmogorov constructed a remarkable example of a
function $f\in L_1(\mathbb{T})$, whose Fourier series diverges a.e. and
even at every point of $\mathbb{T}$ (see \cite{AAP} and \cite{AIN}). After this result some specialists have doubted the validity of the
the Luzin's conjecture, since Kolmogorov's counterexample in $L_1$ was unbounded in any interval, but it was thought to be only a matter of time before a continuous counterexample was found.

This long story has its end only in
1966 when Carleson proved that Luzin's problem has a positive solution. Now this is a fundamental result in mathematical analysis. It may be worth mentioning that even Carleson started (see \cite{RS}) by trying to find a continuous counterexample to the Luzin's conjecture making use of the Blaschke products and Zygmund was very positive about that. But he realized eventually that his approach could not work. Only after this he tried instead to prove Luzin's conjecture.

Soon after two years (in 1968) R.A. Hunt, improving Carleson's
method, showed that the Fourier series of any function $f\in L_p(\mathbb{T}),
p>1,$ converges a.e. on $\mathbb{T}.$ Thus the problem of convergence a.e. of Fourier series is
completely solved in classes $L_p(\mathbb{T})$, $p\geq 1$: if $p> 1$
then we have the convergence and if $p=1$, then the
convergence a.e. fails.

{\bf 2. Partial sums in multidimensional case.} The multidimensional Fourier series has the form
\begin{equation}\label{fs}
\sum\limits_{n\in Z^N}f_ne^{i nx}.
\end{equation}
The definition (\ref{ps1}) of partial sums generalizes to the multidimensional
case in three obvious forms, which are often found in the literature.

Partial sums of series (\ref{fs})
$$
M_kf(x)=\sum\limits_{|n_1|\leq
k}\cdot\cdot\cdot\sum\limits_{|n_N|\leq k} f_ne^{i nx}
$$
is called \textbf{square partial sums}. We say that the above series converges over
squares if $\lim\limits_{k\rightarrow \infty} M_k f(x)$ exists.

Second partial sums have the form
$$
R_{k,\cdot\cdot\cdot,l}f(x)=\sum\limits_{|n_1|\leq
k}\cdot\cdot\cdot\sum\limits_{|n_N|\leq l} f_ne^{i nx},
$$
which is called \textbf{rectangular partial sums}. We say that series (\ref{fs}) converges
over rectangles if
$$
\lim\limits_{\min\{k,...,l\}\rightarrow \infty}
R_{k,\cdot\cdot\cdot,l} f(x)
$$
exists.

Finally, \textbf{spherical partial sums} of series (\ref{fs}) are defined as (see (\ref{SL}))
$$
S_\lambda f(x)=\sum\limits_{|n|^2< \lambda}f_n\,e^{inx}.
$$
If
$\lim\limits_{\lambda\rightarrow \infty} S_\lambda f(x)$
exists, then we say that series (\ref{fs}) converges over spheres.

All these
partial sums are a natural generalization of one dimensional case
(see (\ref{ps1})). But from spectral theory point of view, the
natural  generalization is the spherical partial sums. Note the
eigenfunctions of the periodical Laplace operator are the
functions
$u_n(x)=(2\pi)^{-\frac{N}{2}}e^{inx},$
while the eigenvalues equal $\lambda_n=|n|^2.$ Thus eigenfunction
expansions of function $f$ have the form (\ref{SL}).

{\bf 3. Convergence almost-everywhere.} In 1970 N. R. Tevzadze \cite{T} has extended Carleson's theorem
to double Fourier series with square partial sums. The idea of his proof is very simple. Indeed, let $f\in L_2(\mathbb{T}^2)$
be expanded in a double Fourier series whose quadratic partial sums equal
$$
M_k f (x_1, x_2) = \sum\limits_{|n_1|\leq
k}\sum\limits_{|n_2|\leq k} f_{n_1, n_2}e^{i (n_1x_1+n_2x_2)}.
$$
One can rewrite this sum as a sum of two "one dimensional" partial sums
\begin{equation}\label{M}
M_k f (x_1, x_2) = \sum\limits_{|n_1|\leq
k} A_{n_1}(x_2)e^{i n_1x_1}+\sum\limits_{|n_2|\leq
k} B_{n_2}(x_1)e^{i n_2x_2},
\end{equation}
where the coefficients have the form
$$
A_{n_1}(x_2)=\sum\limits_{|n_2|\leq
|n_1|} f_{n_1, n_2}e^{i n_2x_2}, \quad B_{n_2}(x_1)=\sum\limits_{|n_1|\leq
|n_2|} f_{n_1, n_2}e^{i n_1x_1}.
$$
It is not hard to prove, using Levi's theorem, that the series
$$
\sum\limits_{n_1=-\infty}^{\infty}A^2_{n_1}(x_2)
$$
converge for almost every $x_2\in\mathbb{T}$. This means, that for
almost every $x_2\in \mathbb{T}$ we can consider  the first of the
sums in the right hand side of (\ref{M}) as the partial sums of a
Fourier series of a $L_2$ - function of $x_1$. Therefore by
Carleson's theorem, this sum must converge for almost all $x_1\in
\mathbb{T}$. By the same way one can prove that the second sum in
(\ref{M}) for almost every $x_1$ converges in $x_2$ a.e. on
$\mathbb{T}$, i.e.
$$
\lim\limits_{k\rightarrow\infty}M_kf(x_1, x_2) = f(x_1, x_2)
$$
a.e. on $\mathbb{T}^2$. Note that the convergence to $f$ exactly follows from the fact that $M_k$ converges to $f$ in $L_2(\mathbb{T}^2)$.

In an analogous way one can prove, that
\begin{equation}\label{Ml}
\lim\limits_{k\rightarrow\infty}\widetilde{M}_kf(x_1, x_2) = f(x_1, x_2),
\end{equation}
for almost all $(x_1, x_2)\in \mathbb{T}^2$, where
$$
\widetilde{M}_k f (x_1, x_2) = \sum\limits_{|n_1|\leq
m_1(k)}\sum\limits_{|n_2|\leq m_2(k)} f_{n_1, n_2}e^{i (n_1x_1+n_2x_2)}
$$
and $m_1(k)$ and $m_2(k)$  are arbitrary decreasing sequences of integers.

It should be noted that the proof presented in the above paper
easily can be extended to $N$-dimensional case.

An analogue of the Hunt's theorem was obtained by P.  Sj$\ddot{o}$lin in 1971
\cite{SJ71}.

\textbf{Theorem} (Sj$\ddot{o}$lin).
\emph{If $f\in L_p((\mathbb{T}^N)$,  $p>1$,
then one has $\lim\limits_{k\to\infty}M_kf(x)=f(x)$ a.e. on} $\mathbb{T}^N$.

Note, in fact Sj$\ddot{o}$lin proved this theorem for even wider class of functions $f$.

It is not hard to verify, that if $f_1$  is Kolmogorov's function, then square partial
sums of the function $f(x_1, ..., x_N)=f_1(x_1)\in L_1(\mathbb{T}^N)$ diverge a.e. on $\mathbb{T}^N$.

Thus the problem of convergence a.e. of square partial sums of multiple Fourier series
is completely solved in classes $L_p(\mathbb{T}^N)$, $p\geq 1$: if $p> 1$
then we have the convergence and if $p=1$, then the
convergence a.e. fails. Hence, convergence a.e. of square partial sums of $L_p$-functions
does not depend on the number of dimensions and result has exactly the same formulation for $N>1$ as in one dimensional case.
However, for rectangular partial sums due to C. Fefferman \cite{Fef} we have a completely new phenomena.

\textbf{Theorem} (Charles Fefferman (1971)). \emph{There exists a
function $f\in C((\mathbb{T}^2)$, such that
$$
\overline{\lim\limits_{\min\{k,l\}\to\infty}}|R_{k,l}
f(x)|=+\infty
$$
everywhere on $\mathbb{T}^2$.}

The divergence in Fefferman's theorem is achieved precisely on
partial sums of the form (\ref{Ml}). Since $C(\mathbb{T}^2)\subset
L_2(\mathbb{T}^2)$, at first glance this contradicts equality
(\ref{Ml}). But of course there is no contradiction here. In
Fefferman's example at each point $(x_1, x_2)\in Q\subset \subset
\mathbb{T}^2$ partial sums $\widetilde{M}_kf(x_1, x_2)$ with
$m_1(k)=[kx_1]$ and $m_2(k)=[kx_2]$ diverge at $(x_1, x_2)$, but
according to (\ref{Ml}), converge a.e. on $\mathbb{T}^2$.

Obviously, if $f_1$  is the function from this theorem, then function $f(x_1, ..., x_N)=f_1(x_1, x_2)$ gives a counterexample
for any dimensional case.
Since $C(\mathbb{T}^N)\subset L_p(\mathbb{T}^N)$ for any $p\geq 1$, then for rectangular partial sums
the problem of a.e. convergence has the following solution: for any $p\geq 1$
the
convergence a.e. fails in $L_p(\mathbb{T}^N)$ classes.

It may be worth mentioning that for rectangular partial sums  there are examples of functions
with high smoothness for which the classical principle of localization is not valid
(see \cite{AAP}). Therefore, for these partial sums too (as in case of spherical partial sums)
the principle of generalized localization may become a replacement for the classical principle of localization.

I. Bloshanskiy intensively and in detail studied the generalized
localization for rectangular partial sums (see \cite{AAP}). In
particular he proved \cite{Bl}, that generalized localization of
rectangular partial sums belongs to those properties which
discriminate between two- and three- dimensional Fourier series.
Namely, he showed that generalized localization holds in the
classes $L_p(\mathbb{T}^N)$, $p>1$, for $N=2$ but not for $N\geq
3$.

If $p=1$, generalized localization  for rectangular partial sums
does not hold for any $N\geq 2$. To prove this statement it is
sufficient to consider function $f(x)=f_1(x_1)\cdot f_2(x_2 , ...,
x_N)$, where $f_1$ is Kolmogorov's example of function in
$L_1(\mathbb{T})$ with an unboundedly divergent Fourier series and
$f_2$ is a $L_1$-function which is equal to zero in a neighborhood
of the origin.

\section{Generalized localization principle for the multiple Fourier integrals}

\textbf{1. The Carbery-Soria theorem.} For the spherical partial sums (\ref{SL}) the Luzin's conjecture is still open.
But, as we will see the next section, the problem of generalized localization principle
has a positive solution.

By the definition of Fourier coefficients, one can write
$$
S_\lambda f(x)= \int\limits_{\mathbb{T}^N}\theta(x-y,\lambda) f(y)dy,
$$
where the kernel has the form
$$
\theta(x,\lambda)=(2\pi)^{-N}\sum\limits_{|n|^2<\lambda}e^{inx}.
$$

To investigate convergence of $S_\lambda f(x)$ when
$\lambda\rightarrow \infty$ one has to investigate the asymptotic
behavior of $\theta(x,\lambda)$ for large $\lambda$. Of course,
this is not an easy problem.  But if we change the sum into an
integral, then we have the explicit form of the kernel:
$$
e(x,\sqrt{\lambda})=(2\pi)^{-N}\int\limits_{|\xi|^2<\lambda}e^{ix\xi}d\xi=C\frac{J_{\frac{N}{2}}(|x|\sqrt{\lambda})}{(|x|\sqrt{\lambda})^{N/2}},
$$
where $J_\nu (t)$ is the Bessel function.
On the other hand, $ e(x,\lambda)$ is the kernel of  the integral operator
\begin{equation}\label{EL}
E_\lambda f(x)=\int\limits_{\mathbb{R}^N}e(x-y,\lambda)
f(y)dy=(2\pi)^{-\frac{N}{2}}\int\limits_{|\xi|<\lambda}\hat{f}(\xi)e^{ix\xi}d\xi,
\end{equation}
which is the spherical partial integrals of the multiple Fourier
integral of a function $f\in L_2(\mathbb{R}^N)$. Here a Fourier transform of function $f$ is defined as
$$
\hat{f}(\xi)=(2\pi)^{-\frac{N}{2}}\int\limits_{\mathbb{R}^N} f(x)e^{-ix\xi}dx.
$$

For the spherical partial integrals of multiple Fourier integrals
the generalized localization principle in $L_p(\mathbb{R}^N)$ has been
investigated by many authors from 1983 till 2016
(Sj\"{o}lin, P., Carbery, A., Soria F., Rubio de Francia,
J. L., Vega, L.,  Bastys A., and others (see \cite {SJ} - \cite{AB2}).

In particular,  in the fundamental paper \cite{CAR} the following  theorem has been proved

\textbf{Theorem} (Carbery, Soria (1988)). \emph{Let $f\in L_p(\mathbb{R}^N)$,
$2\leq p<2N/(N-1)$ and $f=0$ on an open set $\Omega\subset \mathbb{R}^N$.
Then }
\begin{equation}\label{E0}
\lim\limits_{\lambda \to \infty} E_\lambda f(x) =  0
\end{equation}
\emph{holds a.e. on }$\Omega$.

An earlier variant of this theorem where $f\in L_2(\mathbb{R}^N)$ was assumed to have a compact support
was obtained by Sj\"{o}lin, \cite{SJ}.

We prove  Sj\"{o}lin's theorem here by the method of Carbery and
Soria; the proof of the above theorem (Carbery and Soria) is
significantly more complicated. It may be worth mentioning that
the original proof of Sj\"{o}lin is complectly different.

Let $f\in L_2(\mathbb{R}^N)$ and $f$ have a compact support. In these
conditions we must prove that the equality (\ref{E0}) holds a.e.
on $\Omega = \mathbb{R}^N \setminus \ supp f$. If $x\in\Omega$ an arbitrary point, then to do this
it suffices to show validity of (\ref{E0}) a.e. on a ball with
center at $x$ and sufficiently small radius $R$, so that this ball
belongs to $\Omega$. Therefore without loss of generality we may
suppose, that $f$ is supported outside of this ball or assuming $\{|x|< R\}\subset \Omega$, $f$ is supported in $\{R\leq |x|\leq A\}$, and
prove convergence to zero of $E_\lambda f(x)$ a.e. on the ball
$\{|x|<r\}$ for any $r<R$. But this statement can be proved by a
standard technique based on the following theorem  (see \cite{ST}).

\begin{thm} Let $ E_\star f(x)=\sup\limits_{\lambda>0}|E_\lambda
f(x)|$ be the maximal operator.  Then for any $r<R$ there exists a
constant $C=C(R,r)$, such that
\begin{equation}\label{MO}
\int\limits_{|x|\leq r}|E_\star f(x)|^2 dx\leq C
\int\limits_{\mathbb{R}^N}|f(x)|^2 dx.
\end{equation}
\end{thm}

First we prove some auxiliary assertions. Let $\chi (x)\in C_0^\infty (\mathbb{R}^N)$ be a radial function, such that
$$
\chi(x)=
\left\{
\begin{array}{l}
0,\,\,\,\,|x|<\frac{1}{3}(R-r),\\
1,\,\,\,\,\frac{2}{3}(R-r)< |x|\leq 2 A.
\end{array}
\right.
$$
If we denote $e_\lambda(x)=e(x,\lambda) \chi(x)$, then we have
$$
E_\lambda f(x)=\int\limits_{\mathbb{R}^N}e_\lambda(x-y)f(y)dy, \,\,for
\,\,all\,\, x,\,\,with \,\,|x|\leq r,
$$
since $f$ is supported in $\{R\leq |x|\leq A\}$. Therefore to
prove estimate (\ref{MO}) it suffices to obtain the inequality
\begin{equation}\label{e}
\int\limits_{\mathbb{R}^N}\sup_{\lambda >\,0}\left|e_\lambda\ast f\right|^2dx\leq
C\int\limits_{\mathbb{R}^N}|f(x)|^2dx.
\end{equation}
Now we need some estimates for the Fourier transform of the
function $e_\lambda(x)$, which we denote by
$\widehat{e}_\lambda(\eta)$. These estimates were obtained in
\cite{CAR} and for the convenience of readers we will give the
proof from \cite{CAR} here.
\begin{lem}\label{ft}
For  $j=0,1$ and for an arbitrary $l\in \mathbb{N}$ there exists a
constant $C_l$, depending on $l, \, r $ and $R$, such that for all
$\lambda\in \mathbb{R}$ and $\eta\in \mathbb{R}^N$ one has
$$
|\frac{d^j}{d\lambda^j}\widehat{e}_\lambda(\eta)|\leq\frac{C_l}{(1+||\eta|-\lambda|)^l}.
$$
\end{lem}
\begin{proof}We shall consider the case $j=0$, the case $j=1$ being handled similarly (see  \cite{CAR}).
By the definition of the Fourier transform we have
$$
\widehat{e}_\lambda(\eta)=(2\pi)^{-\frac{3}{2}N}\int\limits_{\mathbb{R}^N}\big(\int\limits_{|\xi|<\lambda}e^{ix\xi}d\xi\chi(x)\big)e^{-ix\eta}
dx=
(2\pi)^{-N}\int\limits_{|\xi|<\lambda}(2\pi)^{-\frac{N}{2}}\int\limits_{\mathbb{R}^N}\chi(x)e^{-ix(\eta-\xi)}dxd\xi=
$$
$$
=
(2\pi)^{-N}\int\limits_{|\xi|<\lambda}\widehat{\chi}(\eta-\xi)d\xi=(-2\pi)^{-N}\int\limits_{|\eta-\xi|<\lambda}\widehat{\chi}(\xi)d\xi.
$$
First suppose $|\eta|<\lambda$. Since $\chi(0)=0$ and for any
integer $j$ one has the estimate $|\widehat{\chi}(\xi)|\leq
C_j(1+|\xi|)^{-j}$, then
$$
|\widehat{e}_\lambda(\eta)|=|-(-2\pi)^{-N}\int\limits_{|\eta-\xi|\geq\lambda}\widehat{\chi}(\xi)d\xi|\leq
C_j\int\limits_{|\eta-\xi|\geq\lambda}(1+|\xi|)^{-j}d\xi\leq
$$
$$
\leq
C_j\int\limits_{|\xi|>\lambda-|\eta|}(1+|\xi|)^{-j}d\xi\leq\frac{C_l}{(1+||\eta|-\lambda|)^l}.
$$
Similarly, if $|\eta|>\lambda$,
$$
|\widehat{e}_\lambda(\eta)|=|(-2\pi)^{-N}\int\limits_{|\eta-\xi|<\lambda}\widehat{\chi}(\xi)d\xi|\leq
C_j\int\limits_{|\xi|>|\eta|-\lambda}(1+|\xi|)^{-j}d\xi\leq\frac{C_l}{(1+||\eta|-\lambda|)^l}.
$$

\end{proof}
It is not hard to verify that form this lemma one has
\begin{cor}\label{INT} Let $j=0, 1$. Uniformly in $\eta\in\mathbb{R}^N$ we have
$$
\int\limits_0^\infty
|\frac{d^j}{d\lambda^j}\widehat{e}_\lambda(\eta)|^2d\lambda < C.
$$
\end{cor}
Now we are ready to prove estimate (\ref{e}). Since
$$
[e_\lambda\ast f]^2=2\int\limits_0^\lambda (e_t\ast f)
\frac{d}{dt}(e_t\ast f)dt
$$
and $2ab \leq a^2+b^2$
$$
\sup_{\lambda >\,0}\left|e_\lambda\ast
f\right|^2\leq\int\limits_0^\infty|e_t\ast
f|^2dt+\int\limits_0^\infty|\frac{d}{dt}e_t\ast f|^2dt,
$$
or
$$
\int\limits_{\mathbb{R}^N}\sup_{\lambda >\,0}\left|e_\lambda\ast
f\right|^2\leq\int\limits_{\mathbb{R}^N}|\widehat{f}(\eta)|^2\int\limits_0^\infty
|\widehat{e}_t(\eta)|^2dtd\eta
+\int\limits_{\mathbb{R}^N}|\widehat{f}(\eta)|^2\int\limits_0^\infty
|\frac{d}{dt}\widehat{e}_t(\eta)|^2dtd\eta.
$$
Finally, making use of Corollary \ref{INT} and the fact that $f$
is an $L_2$ - function we obtain estimate (\ref{e}).

Thus theorem of Carbery and Soria is proved for functions $f\in
L_2(\mathbb{R}^N)$ with compact support.

Let $A(\xi)=\sum\limits_{|\alpha|= m}a_\alpha \xi^\alpha$ be an
arbitrary elliptic homogeneous polynomial on $\xi\in \mathbb{R}^N$
with constant coefficients $a_\alpha$, i.e. $A(\xi)>0$ for all
$\xi\neq 0$.

With this polynomial $A(\xi)$ we can define a partial sums of the
multiple Fourier integrals as
$$
E(A, \lambda, f)(x)= (2\pi)^{-\frac{N}{2}} \int\limits_{A(\xi)<
\lambda} \hat{f}(\xi)e^{ix \xi} d\xi.
$$

Note, if $A(\xi)=|\xi|^2$, then the partial sums $E(A, \lambda,
f)(x)$ coincides  with the spherical partial sums
$E_{\sqrt{\lambda}} f(x)$.

In the paper \cite{AAA}, using the method of Carbery-Soria
\cite{CAR}, the following theorem is proved.

\begin{thm} Let $A(\xi)$ be an arbitrary elliptic polynomial. Let  $f\in L_p(\mathbb{R}^N)$,
$2\leq p<2N/(N-1)$ and $f=0$ on an open set $\Omega\subset
\mathbb{R}^N$. Then $ \lim_{\lambda \to \infty} E(A, \lambda,
f)(x) =  0 $ holds a.e. on $\Omega$.
\end{thm}
So in case of an arbitrary elliptic polynomial $A(\xi)$ we have
the same result as for the spherical partial sums $E_\lambda
f(x)$. We note that, unlike to the generalized localization, the
conditions for the classical Riemann localization of $E(A,
\lambda, f)(x)$ strictly depends on the geometry of the set
$\{\xi\in \mathbb{R}^N: A(\xi)<1\}$, namely on the number of
nonzero curvatures of the level surface $\{\xi\in \mathbb{R}^N:
A(\xi)=1\}$ (see \cite{AR1} and \cite{AR2}).

\textbf{2. Generalized localization for distributions.} We recall the Schwartz space $S$ is the class $C^\infty(\mathbb{R}^N)$ being equipped
with a family of seminorms
$$
d_{\alpha,\beta}(\phi) = \sup_{x \in \mathbb{R}^N} |x^\alpha D^\beta
\phi(x)|.
$$
We also consider a class of tempered distributions $S'$ defined as
dual to $S$.

Along with $S$ we also consider $\mathcal{E}$, i.e. $C^\infty(\mathbb{R}^N)$
being equipped with a family of seminorms
$$
\rho_{\alpha, K} (\phi) = \sup_{x \in K}|D^\alpha \phi(x)|.
$$
As usual we denote its conjugate space by $\mathcal{E'}$.
Distributions from $\mathcal{E'}$ have the following important
properties:
\begin{itemize}
\item every $f\in \mathcal{E'}$ has a finite support, \item for any $f\in \mathcal{E'}$ one has
$\hat{f}\in C^\infty$ (the Paley-Wienner theorem), \item for every
$f \in \mathcal{E}'$ there is $l\in \mathbb{R}$ such that $f \in H^l$ .
\end{itemize}

\begin{defin}
We say that tempered distribution $f$ belongs to the Sobolev class
$H^l$ if $\hat{f}$ is a regular distribution such that
$$
\|f\|^2_{H^l} = \int |\hat{f}(\xi)|^2 (1+|\xi|^2)^l d\xi < \infty.
$$
\end{defin}
\textbf{Example.} Consider $D^\alpha \delta (x)\in \mathcal{E}'$, where $\delta (x)$ is the Dirac delta-function. We have $$\hat{(D^\alpha
\delta)} (\xi)=(-1)^{|\alpha|}\xi^\alpha.$$ Therefore $D^\alpha
\delta (x)\in H^{-l}$ with $l>\frac{N}{2}+|\alpha|$.

Since for any distribution $f \in \mathcal{E}'$ its Fourier transform is infinitely differentiable function,
then we can define the Riesz means of order $s\geq 0$ of the spherical partial integrals as follows
$$
E^s_\lambda
f(x)=(2\pi)^{-\frac{N}{2}}\int\limits_{|\xi|<\lambda}\big(1-\frac{|\xi|^2}{\lambda^2}\big)^s\hat{f}(\xi)e^{ix\xi}d\xi.
$$

The authors of the paper \cite {AB2} obtained the necessary and sufficient conditions
for the generalized localization of the Riesz means of distributions to be hold.
These conditions are formulated in the following two theorems.

\begin{thm}\label{D}Let $f \in \mathcal{E}'$$ \cap H^{-l}$, $l \geq 0$. Then for
$s\geq l$, equality
$$
\lim_{\lambda \to \infty} E^s_\lambda f(x) = 0
$$
holds true a.e. on $\mathbb{R}^N\setminus supp \ f$.
\end{thm}
\begin{thm}\label{N} For any $s$ and $l$ such that $0\leq s<l$ there are $f \in
H^{-l}\cap \mathcal{E}'$ and a set $K$ of positive measure such
that $supp \ f\cap K = \emptyset$ and
$$
\limsup_{\lambda\to \infty} |E^s_\lambda f(x)| = \infty, \ a.e.
x\in K.
$$
\end{thm}

In Theorem \ref{D} it is actually proved that the action on a distribution of the Riesz means of order $s$
makes it roughly $s$ - smoother.

Observe, as mentioned above for $E_\lambda^0 f(x)\equiv E_\lambda f(x)$ the generalized localization holds true for
$L_2$-functions. The following theorem (proved in \cite{AB2}) states, that for a negative order Riesz means
the generalized localization fails.

\begin{thm}
\label{nr} For any $\epsilon>0$
there are compactly supported $g \in L_2(\mathbb{R}^N)$ and set $K:
mes(K)>0$ and $supp \ g \cap K = \emptyset $ such that
$$
\sup_{\lambda>0} |E^{-\epsilon}_\lambda g(x)| = \infty, x \in K.
$$
\end{thm}

The arguments of the proof  here essentially consist in reducing the problem
to the generalized localization of spherical means $E_\lambda g(x)$ on $L_p$ with $p<2$, for which there is a counterexample due to Bastys \cite{BAS}.

We note, that the proof of Theorem \ref{N} is essentially based on
Theorem \ref{nr}.

\textbf{3. Generalized localization for continuous wavelet
decompositions.}

In the paper \cite{AF}, the authors  studied the generalized
localization principle for multidimensional spherically symmetric
continuous wavelet decompositions. It should be noted that the
pointwise convergence of one-dimensional and multidimensional
continuous wavelet decompositions has been investigated by many
authors (a detailed survey of their work can be found in the paper
\cite{AB3} by Ashurov and Butaev). In this case, to ensure the
almost-everywhere convergence of multidimensional (and even
one-dimensional) spherically symmetric continuous wavelet
decompositions, it is necessary to impose an additional
restriction on the rate of decay at infinity of the corresponding
wavelets. In the above paper \cite{AF}, using the result of
Carbery-Soria \cite{CAR}, it was shown that the generalized
localization principle holds for multidimensional spherically
symmetric continuous wavelet decompositions  (with completely
arbitrary wavelets) of functions $f\in L_p(\mathbb{R}^N)$,
provided $2\leq p<2N/(N-1)$.

\section{Generalized localization principle for the multiple Fourier series}

\textbf{1. The main result.} If we turn back to the multiple Fourier series (\ref {SL}) and
consider the classes $L_p(\mathbb{T}^N)$ when $1\leq p<2$, then as A.
Bastys \cite{BAS}  has proved, following Fefferman in making use
of the Kakeya's problem, that the generalized localization for
$S_\lambda$ is not valid, i.e. there exists a function $f\in
L_p(\mathbb{T}^N)$, such that on some set of positive measure, contained in
$\mathbb{T}^N\backslash supp f,$ we have
$$
\overline{\lim\limits_{\lambda\to\infty}}|S_\lambda f(x)|=+\infty.
$$
It should be noted that in \cite{BAS} this
result is also proved for the spherical partial integrals
$E_\lambda f(x)$.

The following result gives a complete solution of the generalized localization
problem for the spherical partial sums of the multiple Fourier series. It may be worth mentioning
that this problem was first formulated in 1976
in a review paper \cite{AIN}.

\begin{thm}\label{MAIN}
Let $f\in L_2(\mathbb{T}^N)$ and $f=0$ on an open set $\Omega\subset \mathbb{T}^N$.
Then the equality $\lim\limits_{\lambda\to\infty}S_\lambda f(x)=0$ holds a.e. on $\Omega$.
\end{thm}

Thus the problem of generalized localization for $S_\lambda$ is
completely solved in classes $L_p(\mathbb{T}^N)$, $p\geq 1$: if $p\geq2$
then we have the generalized localization and if $p<2$, then the
generalized localization fails.

Recall that in order for classical Riemann localization to take
place for $S_\lambda f$, the expandable function $f$ must have
$\frac{N-1}{2}$ "derivatives", i.e. $f$ must belong to the Sobolev
class $H^l(\mathbb{T}^N)$ with $l\geq \frac{N-1}{2}$ (see
\cite{AIN}).

Let us introduce the
maximal operator
$$ S_\star f(x)=\sup\limits_{\lambda>0}|S_\lambda
f(x)|. $$

The proof of Theorem \ref{MAIN} is based on the following estimate
of this operator.

\begin{thm}\label{MAX}
Let $\Omega$ be an open subset of $\mathbb{T}^N$. Then for any
compact set $K\subset \Omega$ there exists a constant $C_K>0$ such
that for any function $f\in L_2(\mathbb{T}^N)$ with $\text{supp}\,
f\subset \mathbb{T}^N\setminus \Omega$ one has
\begin{equation}\label{max}
\|S_\star f(x)\|_{L_2(K)}\ \leq\ C_K\|f\|_{L_2(\mathbb{T}^N)}.
\end{equation}
\end{thm}

The formulated  theorems are easily transferred to the case of non-spherical partial sums of multiple Fourier series (see \cite{AAA}, \cite{AB}).

We should also note, that in the remarkable paper of C. Kenig and
P. Tomas \cite{KP} the authors  proved, by making use of
transference techniques, the equivalence of convergence a.e. of
spherical partial sums of multiple Fourier series and integrals.
It may be worth mentioning that Theorem \ref{MAIN} on the
generalized localization does not follow from the Carbery - Soria
theorem  on Fourier integrals  \cite{CAR}, by application of the
C. Kenig and P. Tomas theorem \cite{KP}. The reason here is that,
we integrate  on the left hand side of (\ref{max}) only over a
compact $K$, where the function vanishes (not over whole domain
$\mathbb{T}^N$), and therefore we can not apply the duality
technique, which is the key step in the proof of C. Kenig and P.
Tomas. Nevertheless, to prove Theorem \ref{MAIN} we have used many
of original ideas from A. Carbery and F. Soria \cite{CAR}.

\textbf{2. Auxiliary assertions.} It is not hard to verify, that Theorem \ref{MAX} is an easy
corollary of the following theorem:
\begin{thm}\label{MAX1}
Let $f\in L_2(\mathbb{T}^N)$ and $f=0$ on the ball $\{|x|<R\}$, $R<1$.
Then for any $r<R$ there exists a constant $C=C(R,r)$, such that
\begin{equation}\label{MES}
\int\limits_{|x|\leq r}|S_\star f(x)|^2 dx\leq C
\int\limits_{\mathbb{T}^N}|f(x)|^2 dx.
\end{equation}
\end{thm}

The proof of Theorem \ref{MAX1} can be found in \cite{AR}. We
present here the main ideas of the proof.

So we assume that $f=0$ on the fixed ball $\{|x|<R\}$ and fix a number $r<R<1$.

Let $\chi_b(t)$ be the characteristic function of the segment
$[0,b]$. We denote by $\varphi_1(t)$ a smooth function with
$\chi_{(R-r)/3}(t)\leq \varphi_1(t)\leq \chi_{2(R-r)/3}(t)$ and
put $\varphi_2(t)=1-\varphi_1(t)$. Now we define a new function
$\psi(x)$ as follows: $\psi(x)=\varphi_2(|x|)$, when $x\in \mathbb{T}^N$
and otherwise it is a $2\pi$ - periodical on each variable $x_j$
function.

Let us denote $\theta_\lambda(x)=\theta(x,\lambda)\psi(x)$. Then we
have
$$
S_\lambda f(x)=\int\limits_{\mathbb{T}^N}\theta_\lambda(x-y)f(y)dy, \,\,for
\,\,all\,\, x,\,\,with \,\,|x|\leq r,
$$
since $f$ is supported in $\{|x|\geq R\}$. Therefore, if we denote
the last integral by $\theta_\lambda\ast f$, then to prove the
estimate (\ref{MES}) it suffices to obtain the inequality
\begin{equation}\label{theta}
\int\limits_{\mathbb{T}^N}\sup_{j >\,0}\left|\theta_j\ast f\right|^2dx\leq
C\int\limits_{\mathbb{T}^N}|f(x)|^2dx,
\end{equation}
where $\sup$ is taken over all positive integers.

\begin{rem}Of course, by analogy with the multiple Fourier
integrals (see (\ref{EL})), in partial sums (\ref{SL}), one could
sum over the set $\{|n|<\lambda\}$. But then we would have to take
the $sup$ in (\ref{theta}) over all positive real numbers, instead
of all positive integers.
\end{rem}

Now we need some estimates for the Fourier coefficients of the
function $\theta_j(x)$, which we denote by $(\theta_j)_n$, $j\in
\mathbb{N}$ (the set of positive integers), $n\in \mathbb{Z}^N$.

\begin{lem}\label{coef2}
For an arbitrary $l\in \mathbb{N}$ there exists a constant $C_l$,
depending on $l, \, r $ and $R$, such that for all $j\in
\mathbb{N}$ and $n\in \mathbb{Z}^N$ one has
\begin{equation}\label{Coef2}
|(\theta_j)_n|\leq\frac{C_l}{(1+||n|-\sqrt{j}|)^l}.
\end{equation}
\end{lem}

Let $\{\psi_m\}$ be the Fourier coefficients of function
$\psi(x)$: $\psi_m=(2\pi)^{-N}\int_{\mathbb{T}^N}
\psi(y)e^{-imy}dy$. Then for any integer $q\geq 0$ there exists a
constant $c_q$, depending on $(R-r)$, such that
$$
|\psi_m|\leq\frac{c_q}{(1+|m|)^q}.
$$
Lemma is an easy consequence of this estimate (see the proof of Lemma \ref{ft}).

It should be noted the importance of considering the function $\theta_j(x)$ instead of $\theta(x, j)$.
The Fourier coefficients $(\theta(x, j))_n$ of function $\theta(x, j)$ have the form
$$
(\theta(x, j))_n=
\left\{
\begin{array}{l}
1,\,\,\,\,j>|n|,\\
0,\,\,\,\,j\leq|n|.
\end{array}
\right.
$$
These numbers do not decrease as $j\rightarrow \infty$, while the numbers $(\theta_j)_n$ have a sufficiently strong decrease (see estimate (\ref{Coef2})).

Let $(\Theta_j)_n=(\theta_{j+1})_n-(\theta_{j})_n$, that is,
$$
(\Theta_{j})_n= (2\pi)^{-N}\sum\limits_{
|m|^2=j}\psi_{m-n}=(2\pi)^{-N}\sum\limits_{
|n-m|^2=j}\psi_{m}
$$
(if the Diophantine equation $|m-n|^2=j$ does not have a solution,
then $(\Theta_{j})_n=0$). Using the same arguments as above, one
can prove the following estimate.

\begin{lem}\label{bigl}

For any $k$ and $l\in\mathbb{N}$, there exists a constant $C_l$
such that
\begin{equation}\label{Big}
\sum\limits_{k\leq\sqrt{j}<k+1}|(\Theta_j)_n|^2\leq
\frac{C_l}{(1+||n|-k|)^l}.
\end{equation}
\end{lem}

If we sum this estimate with respect to $k$ from $0$ to $\infty$,
then we obtain the next corollary.

\begin{cor}\label{LBig} Uniformly in $n\in\mathbb{Z}^N$ one has

$$\sum\limits_{j=0}^\infty|(\Theta_{j})_n|^2=\sum\limits_{k=0}^\infty\sum\limits_{k\leq\sqrt{j}<k+1}|(\Theta_j)_n|^2\leq C.$$

\end{cor}

Now we try to estimate the same sum for $(\theta_{j})_n$. Let
$k\leq\sqrt{j}< k+1$, i.e. $j=k^2+p,$ $p=0, 1, ..., 2k$. From the
estimate (\ref{Coef2}) we have
$$
\sum\limits_{k\leq\sqrt{j}<k+1}|(\theta_j)_n|^2
=\sum\limits_{p=0}^{2k}|(\theta_{k^2+p})_n|^2\leq \frac{C_l\cdot
k}{(1+||n|-k|)^l},
$$
therefore
\begin{equation}\label{W}
\sum\limits_{j=0}^\infty|(\theta_{j})_n|^2=\sum\limits_{k=0}^\infty\sum\limits_{k\leq\sqrt{j}<k+1}|(\theta_j)_n|^2\leq
C\cdot |n|.
\end{equation}

But this estimate is not sufficient for our purpose. Indeed, let
$\Theta_j(x)=\theta_{j+1}(x)-\theta_j(x)$.  Then $\theta_{j+1}\ast
f+\theta_j\ast f=2\, \theta_j\ast f+\Theta_j\ast f$. Note the
Fourier coefficients of the function $\Theta_j(x)$ are the numbers
$(\Theta_j)_n$, introduced above.

Now we use the idea of the proof of estimate (\ref{e}). If for a
sequence of numbers $\{F_q\}$ we have $F_0=0$, then
$$
F_q^2=\sum\limits_{j=0}^{q-1}[F_{j+1}-F_j][F_{j+1}+F_j],\,\, q\geq
1.
$$
Hence
$$
[\theta_q\ast f]^2 =\sum\limits_{j=0}^{q-1}[\Theta_j\ast
f]^2+2\,\sum\limits_{j=0}^{q-1} [\Theta_j\ast f] [\theta_j\ast f],
$$
or
\begin{equation}\label{W2}
\sup_{q >\,0}\left|\theta_q\ast
f\right|^2\leq\sum\limits_{j=0}^{\infty}|\Theta_j\ast f|^2+
2\,\sum\limits_{k=0}^{\infty}\sum\limits_{p=0}^{2k}|\Theta_{k^2+p}\ast
f||\theta_{k^2+p}\ast f|.
\end{equation}
Integrating over $T^N$ and making use of the inequality $2ab\leq
a^2+b^2$ one has
\begin{equation}\label{sum}
\int\limits_{T^N}\sup_{q >\,0}\left|\theta_q\ast f\right|^2\leq
2\sum\limits_{n}|f_n|^2\sum\limits_{j=0}^\infty|(\Theta_{j})_n|^2+
\sum\limits_{n}|f_n|^2\sum\limits_{j=0}^\infty|(\theta_{j})_n|^2.
\end{equation}
By virtue of Corollary \ref{LBig} and the fact that $f$ is an
$L_2$ - function, we can estimate first sum on the right hand side
of (\ref{sum}) by
$C\sum\limits_n|f_n|^2=C\int\limits_{\mathbb{T}^N}|f(x)|^2 dx. $
But the second sum we can estimate, using (\ref{W}), only by
$C\sum\limits_n|f_n|^2 |n|$. Thus we must have a better estimate
for the sum (\ref{W}).

To prove Theorem \ref{MAX1} we need the same estimate as
(\ref{Big}) for the sum
$\sum\limits_{p=0}^{2k}(p+1)^2|(\Theta_{k^2+p})_n|^2$ (that is,
the value under the sum  (\ref{Big})  is multiplied by $(p+1)^2$).
To obtain this estimate we must show, that
$|(\Theta_{k^2+p})_n|^2$ decreases fast enough in $p$.
Unfortunately this is not always the case. Nevertheless if we
change the order of summation in the last sum (since the summands
are positive, this is always possible), say first take the sum
over some sets $Q_q^k$, $q =0, 1, ...$ of parameters $p$, and then
take the sum over $q =0, 1, ...$, we can succeed to prove the
necessary estimate.

\begin{lem}\label{Q}
There exist sets $Q_q^k$, $q =0, 1, \cdot\cdot\cdot, 2k-1$,  of
integers $p$, $0\leq p\leq 2k$, such, that for any $l$, one has an
estimate with a constant $C_l$:
\begin{equation}\label{SBig}
\sum\limits_{q=0}^{2k-1}(q+1)^2\sum\limits_{p\in
Q_q^k}|(\Theta_{k^2+p})_n|^2\leq \frac{C_l}{(1+\sqrt{||n|-k|})^l}.
\end{equation}

\end{lem}

Our next aim is to define the sets $Q_q^k$, $q =0, 1, ...$ with
above property.

Denote by $y_0$ the intersection point of the ball $\{x\in
\mathbb{R}^N: |x-n|\leq k+1\}$ with the straight line $On$ that
passes through the origin and point $n$ (the nearest one to the
origin). Let $T_{y_0}$ be the tangential hyperplane to the ball
$\{x\in \mathbb{R}^N: |x-n|\leq k+1\}$ at the point $y_0$. Let
$B_0:=\{y\in T_{y_0}: |y-y_0| < 1\}$ and $B_j:=\{y\in T_{y_0}:
\sqrt{j}\leq |y-y_0| < \sqrt{j+1}\,\}$, where $j=1,2, ..., 2k-1$.
Let $C^k_j$, $j=0, 1, ..., 2k-1$, be the $N-$ dimensional
cylinders with the base $B_j$ and with the axis parallel to $On$
and the length $|n|$. Consider the ring $K=\{x\in \mathbb{R}^N:
k\leq |x-n| < (k+1)\}$ and divide it in to the following sets:
$P_j^k = K\cap C_j^k$, $j=0, 1, ..., 2k-1$.

Let us define the sets $Q_q^k$, $q =0, 1, ..., 2k-1$,
as follows. Let $Q_q^k$ be the set of those integers $p$, $0\leq
p\leq 2k$, for which the Diophantine equation $|m-n|^2=k^2+p$ has a solution in $P_q^k$. If $P_q^k$
does not contain any of solutions of equation $|m-n|^2=k^2+p$, for any $p$,
then we assign to the set $Q_q^k$ one of those parameters $p$ that are not included in the previous sets $Q_j^k$,
$j =0, 1, ..., q-1$. If there are no such
 $p's$ left, then we define $Q_j^k$, $j =q, q+1, ...,
2k-1$ as the empty set.

In the proof of Lemma \ref{Lsmall} we need to know  how many at
most parameters $p$ does the set $Q_q^k$ contain. Since the sets
$P_q^k$ are sufficiently small, then there should not be many such
parameters. Indeed, the length of the projection of $P_q^k$ on the
axis of $Ox_1$ is at most $2\sqrt{q+1}$. Consequently, if, for a
fixed $p$, there is a solution of the Diophantine equation
$|m-n|^2=k^2+p$, provided $m\in P_q^k$, then the first coordinates
$m_1$ of the numbers $m$, take at most $[2\sqrt{q+1}\,]$ ($[a]$ is
the integer part of the number $a$) different values. When $p$
varies from $0$ to $2k$, each of these numbers $m_1$ can repeat at
most $2^{N-1}$ times (that is the number of all vectors $m=(m_1,
m_2,...,m_N)\in \mathbb{Z}^N$ within unit cube with fixed $m_1$).
Hence each set $Q_q^k$ has at most $2^N[\sqrt{q+1}]$ parameters
$p$ with the above property (one can give a better estimate for
this number, but for our purpose this estimate is sufficient).

If we denote $S_p = \{m\in \mathbb{Z}^N: |m-n|^2 = k^2+p\}$ ($p =
0, 1, ..., 2k$), then with this choice of $Q_q^k$ we have the
following estimates.

If $|n|\geq k+1$, then
$$
\min\limits_{m\in S_p,\,p\in Q_q^k}|m|\geq \sqrt{(|n|-k-1)^2+q}.
$$
If $k< |n|< k+1$, then
$$
\min\limits_{m\in S_p,\,p\in Q_q^k}|m|\geq \sqrt{q}.
$$
If $|n|\leq k$, then
$$
\min\limits_{m\in S_p,\,p\in Q_q^k}|m|\geq
\frac{1}{2}\sqrt{(|n|-k)^2+q}.
$$

Lemma \ref{Q} is a consequence of these estimates.

\begin{cor}\label{Sbigl} Uniformly in $n\in\mathbb{Z}^N$, one has
$$
\sum\limits_{k=0}^{\infty}\sum\limits_{q=0}^{2k-1}(q+1)^2\sum\limits_{p\in
Q_q^k}|(\Theta_{k^2+p})_n|^2\leq C.
$$

\end{cor}

Now we turn back to the Fourier coefficients $(\theta_{j})_n$.
From Lemma \ref{coef2} we have the following estimate.

\begin{lem}\label{Lsmall}

Uniformly in $n\in\mathbb{Z}^N$, one has
$$
\sum\limits_{k=0}^{\infty}\sum\limits_{q=0}^{2k-1}(q+1)^{-2}\sum\limits_{p\in
Q_q^k}|(\theta_{k^2+p})_n|^2\leq C.
$$

\end{lem}

\begin{proof} As mentioned above, each $Q_q^k$ has at most $2^N[\sqrt{q+1}]$ parameters $p$. Therefore, by virtue of Lemma \ref{coef2},
one has
$$
\sum\limits_{k=0}^{\infty}\sum\limits_{q=0}^{2k-1}(q+1)^{-2}\sum\limits_{p\in
Q_q^k}|(\theta_{k^2+p})_n|^2\leq
\sum\limits_{k=0}^{\infty}\frac{C_l}{(1+||n|-k|)^l}\sum\limits_{q=0}^{2k-1}\frac{2^N[\sqrt{q+1}]}{(q+1)^{2}}\leq
C.
$$
\end{proof}

\textbf{3. Proofs of Theorems.} Now we rewrite estimate (\ref{W2})
as
$$
\sup_{q\in\mathbb{N}}\left|\theta_q\ast
f\right|^2\leq\sum\limits_{j=0}^{\infty}|\Theta_j\ast
f|^2+2\,\sum\limits_{k=0}^{\infty}\sum\limits_{q=0}^{2k-1}\sum\limits_{p\in
Q_q^k}|\Theta_{k^2+p}\ast f|(q+1)|\theta_{k^2+p}\ast f|(q+1)^{-1}.
$$

Integrating over $\mathbb{T}^N$ and making use of Corollaries
\ref{LBig}, \ref{Sbigl} and Lemma \ref{Lsmall}, we have

$$
\int\limits_{\mathbb{T}^N}\sup_{q\in\mathbb{N}}\left|\theta_q\ast
f\right|^2\leq\sum\limits_{n}|f_n|^2\sum\limits_{j=0}^\infty|(\Theta_{j})_n|^2+
$$
$$
+\sum\limits_{n}|f_n|^2\sum\limits_{k=0}^{\infty}\sum\limits_{q=0}^{2k-1}(q+1)^2\sum\limits_{p\in
Q_q^k}|(\Theta_{k^2+p})_n|^2+
$$
$$
+\sum\limits_{n}|f_n|^2\sum\limits_{k=0}^{\infty}\sum\limits_{q=0}^{2k-1}(q+1)^{-2}\sum\limits_{p\in
Q_q^k}|(\theta_{k^2+p})_n|^2\leq
C\sum\limits_n|f_n|^2=C\int\limits_{\mathbb{T}^N}|f(x)|^2 dx.
$$

Thus, the estimate (\ref{theta}) and, consequently,  Theorem
\ref{MAX} is proved.

Theorem \ref{MAIN} can be proved by a
standard technique based on Theorem \ref{MAX} (see \cite{ST}). Indeed, let $f\in L_2(\mathbb{T}^N)$ and $f=0$
on $\Omega\subset \mathbb{T}^N$. We must prove
$$\lim_{\lambda \to \infty} S_\lambda f(x) =  0\,\, a.e.\,\,on\,\,\Omega.$$
Let
$K\subset \Omega$ be an arbitrary compact. Then for an arbitrary $\delta>0$
there exist two $2\pi$-periodical on each argument functions
$f_1$ and $f_2$ with the following properties: $f=f_1+f_2$ and $f_1=f_2=0$ on  $\Omega$;
$f_1\in C^\infty(\mathbb{T}^N)$ and
$||f_2||_{L_2(\mathbb{T}^N)}<\delta$.
Moreover for an arbitrary $\varepsilon>0$ and $\lambda >
\lambda_0(\varepsilon)$ one has $|S_\lambda
f_1(x)|<\frac{\varepsilon}{2}$ on $\Omega$ and $M_\varepsilon = mes
\{x\in K\,\, |,\,\, |S_\star
f_2(x)|>\frac{\varepsilon}{2}\}<\varepsilon$ (this inequality
follows from Theorem \ref{MAX}).

Therefore, if $\lambda > \lambda_0(\varepsilon)$ then on
$K\setminus M_\varepsilon$ one has
$$
|S_\lambda f|\leq |S_\lambda f_1|+|S_\star f_2|<\varepsilon.
$$
Obviously,
this means $S_\lambda f(x)\rightarrow 0$ a.e. on $\Omega$.

\section{ Acknowledgement} The author conveys thanks to Sh. A. Alimov
for discussions of this result and gratefully acknowledges S. Umarov (University New Haven, USA) for
support and hospitality.

 The author was supported by Foundation for Support of Basic Research of the Republic of Uzbekistan
 (project number is OT-F4-88).

\

\

\bibliographystyle{amsplain}

\begin{thebibliography}{10}

\bibitem {AAP} Alimov, Sh.A., Ashurov, R.R., Pulatov, A.K.: Multiple
Fourier Series and Fourier Integrals. Commutative Harmonic
Analysis, vol. IV, pp. 1–97. Springer, Berlin (1992)

\bibitem{AIN} Alimov, Sh.A., Il'in, V.A., Nikishin, E.M.: Convergence problems
of multiple trigonometric series and spectral decompositions. I. Russian Math. Surveys, \textbf{31}, 29-86 (1976)

\bibitem{MN} Mityagin, B.S.,  Nikishin, E.M.: On almost-everywhere divergence of Fourier series.
Sov. Math. Dokl. \textbf{14}, 677-680, (1973)

\bibitem{IL} Il'in, V.A.:  On a generalized interpretation of the
principle of localization for Fourier series with respect to
fundamental systems of functions, Sib. Mat. Zh., \textbf{9}, 1093-1106
(1968)


\bibitem{AR} Ashurov, R., R.: Generalized localization
for spherical partial sums of multiple Fourier series, J. Fourier Anal. Appl., 1-10, 2019, /doi.org/10.1007/s00041-019-09697-7


\bibitem{RS} Raussen, M., Skau, Ch.: Interview with Abel Prize recipient Lennart Carleson.
Notices of the American Mathematical Society, \textbf{54} (2): 223-229 (2007)

\bibitem{T} Tevzadze, N. R.: On the convergence of the double Fourier series of a square summable function. Soobshch. Akad. Nauk Gruz. SSR. \textbf{2}, 277-279 (1970)

\bibitem{SJ71}  Sj\"{o}lin, P.: Convergence almost-everywhere of certain singular integrals, Ark. Mat. \textbf{1}, 65-90 (1971)

\bibitem{Fef} Fefferman, C.: On the divergence of multiple Fourier series, Bull. Am. Math. Soc. \textbf{77}, 191-195 (1971)

\bibitem{Bl} Bloshanskiy, I. L.: On the uniform convergence of multiple triginometric series and Fourier integrals,
Math. Notes \textbf{18}, 675-684 (1975)


\bibitem{SJ}  Sj\"{o}lin, P.: Regularity and integrability of
spherical means, Monatsh. Math. \textbf{96}, 277-291 (1983)


\bibitem{BAS2} Bastys, A.J.: The generalized localization principle for
an N-fold Fourier integral, Sov. Math. Dokl., \textbf{278}, 777-778
(1984)

\bibitem{CAR} Carbery, A., Soria F.: Almost everywhere
convergence of Fourier integrals for functions in Sobolev spaces,
and an $L_2$-localization principle, Revista Mat. Iberoamericana
\textbf{4}, 319- 337 (1988)

\bibitem{CFV} Carbery, A., Rubio de Francia, J. L., Vega, L.: Almost
everywhere summability of Fourier integrals, J. London Math. Soc.
\textbf{38}, 513-524 (1988)


\bibitem{BAS} Bastys, A.J.: Generalized localization of Fourier
series with respect to the eigenfunctions of the Laplace operator
in the classes Lp, Litovskii Matematicheskii Sbornik, \textbf{31}, 387-405
(1991)




\bibitem{CRS} Carbery, A., Romera, E., Soria, F.: Radial weights and
mixed norm inequalities for the disc multiplier, J. Funct. Anal.
\textbf{109}, 52-75 (1992)

\bibitem{CAR2} Carbery, A. Soria, F.: Pointwise Fourier inversion and
localization in $R^n$,     J. Fourier Anal.
Appl., \textbf{3}, Special Issue, 847-858 (1997)

\bibitem{AB2} Ashurov, R.R., Butaev, A.: On Generalized Localization of Fourier
Inversion for Distributions, Contemporary Mathematics, USA, \textbf{672}, 33-50  (2016)

\bibitem{AF} Ashurov, R. R., Faiziev, Yu. E.: Generalized localization principle for
continuous wavelet decompositions, Math. Notes \textbf{106}, 75-81
(2019)

\bibitem{AB3} Ashurov, R.R.,  Butaev, A.: On pointwise convergence of continuous wavelet transforms, Uzbek. Math. J.,
\textbf{1}, 4–26 (2018).



\bibitem{ST} Stein,  E.M., Weiss, G.: Itroduction to Fourier analysis
on Euclidean spaces, Princeton University Press, Princeton, New
Jersey, (1971)



\bibitem{AAA} Ashurov,R. R., Ahmedov,A., Ahmad Rodzi b. Mahmud,: The generalized localization
 for multiple Fourier integrals, J. Math. Anal. Appl., \textbf{ 371 } 832-841(2010)

\bibitem{AR1} Ashurov, R., R.: On conditions for localization of spectral expansions corresponding to elliptic operators with constant
coeffitients, Math. Notes, \textbf{33}, 434-439 (1983)

\bibitem{AR2} Ashurov, R., R.: On localization conditios for multiple triginometric Fourier series, Sov. Math. Dokl. \textbf{31}, 496-499 (1985)




\bibitem{AB} Ashurov, R.R., Butaev, A.: On the Pinsky phenomenon, J. Fourier Anal.
Appl. \textbf{16}, 804-812 (2010)



\bibitem{KP} Kenig, C. E., Tomas, P.A.: Maximal operators defined by Fourier multipliers// Studia Math. \textbf{68},  79-83 (1980).






\end{thebibliography}

\end{document}